\documentclass[12pt,reqno]{amsart}
\usepackage{amscd,amssymb,amsmath,mathabx,etex,float}
\newtheorem{thm}[equation]{Theorem}
\numberwithin{equation}{section}
\newtheorem{cor}[equation]{Corollary}

\newtheorem{lem}[equation]{Lemma}

\newtheorem{defin}[equation]{Definition}

\newtheorem{prop}[equation]{Proposition}

\begin{document}
\raggedbottom \voffset=-.7truein \hoffset=0truein \vsize=8truein
\hsize=6truein \textheight=8truein \textwidth=6truein
\baselineskip=18truept

\def\mapright#1{\ \smash{\mathop{\longrightarrow}\limits^{#1}}\ }
\def\mapleft#1{\smash{\mathop{\longleftarrow}\limits^{#1}}}
\def\mapup#1{\Big\uparrow\rlap{$\vcenter {\hbox {$#1$}}$}}
\def\mapdown#1{\Big\downarrow\rlap{$\vcenter {\hbox {$\ssize{#1}$}}$}}
\def\mapne#1{\nearrow\rlap{$\vcenter {\hbox {$#1$}}$}}
\def\mapse#1{\searrow\rlap{$\vcenter {\hbox {$\ssize{#1}$}}$}}
\def\mapr#1{\smash{\mathop{\rightarrow}\limits^{#1}}}
\def\ss{\smallskip}
\def\ar{\arrow}
\def\vp{v_1^{-1}\pi}
\def\at{{\widetilde\alpha}}
\def\sm{\wedge}
\def\la{\langle}
\def\ra{\rangle}
\def\on{\operatorname}
\def\ol#1{\overline{#1}{}}
\def\spin{\on{Spin}}
\def\cat{\on{cat}}
\def\lbar{\ell}
\def\qed{\quad\rule{8pt}{8pt}\bigskip}
\def\ssize{\scriptstyle}
\def\a{\alpha}
\def\bz{{\Bbb Z}}
\def\Rhat{\hat{R}}
\def\im{\on{im}}
\def\ct{\widetilde{C}}
\def\ext{\on{Ext}}
\def\sq{\on{Sq}}
\def\eps{\epsilon}
\def\ar#1{\stackrel {#1}{\rightarrow}}
\def\br{{\bold R}}
\def\bC{{\bold C}}
\def\bA{{\bold A}}
\def\bB{{\bold B}}
\def\bD{{\bold D}}
\def\bh{{\bold H}}
\def\bQ{{\bold Q}}
\def\bP{{\bold P}}
\def\bx{{\bold x}}
\def\bo{{\bold{bo}}}
\def\si{\sigma}
\def\Vbar{{\overline V}}
\def\dbar{{\overline d}}
\def\wbar{{\overline w}}
\def\Sum{\sum}
\def\tfrac{\textstyle\frac}
\def\tb{\textstyle\binom}
\def\Si{\Sigma}
\def\w{\wedge}
\def\equ{\begin{equation}}
\def\AF{\operatorname{AF}}
\def\b{\beta}
\def\G{\Gamma}
\def\D{\Delta}
\def\L{\Lambda}
\def\g{\gamma}
\def\k{\kappa}
\def\psit{\widetilde{\Psi}}
\def\tht{\widetilde{\Theta}}
\def\psiu{{\underline{\Psi}}}
\def\thu{{\underline{\Theta}}}
\def\aee{A_{\text{ee}}}
\def\aeo{A_{\text{eo}}}
\def\aoo{A_{\text{oo}}}
\def\aoe{A_{\text{oe}}}
\def\vbar{{\overline v}}
\def\endeq{\end{equation}}
\def\sn{S^{2n+1}}
\def\zp{\bold Z_p}
\def\cR{{\mathcal R}}
\def\P{{\mathcal P}}
\def\cF{{\mathcal F}}
\def\cQ{{\mathcal Q}}
\def\cj{{\cal J}}
\def\zt{{\bold Z}_2}
\def\bs{{\bold s}}
\def\bof{{\bold f}}
\def\bq{{\bold Q}}
\def\be{{\bold e}}
\def\Hom{\on{Hom}}
\def\ker{\on{ker}}
\def\kot{\widetilde{KO}}
\def\coker{\on{coker}}
\def\da{\downarrow}
\def\colim{\operatornamewithlimits{colim}}
\def\zphat{\bz_2^\wedge}
\def\io{\iota}
\def\Om{\Omega}
\def\Prod{\prod}
\def\e{{\cal E}}
\def\zlt{\Z_{(2)}}
\def\exp{\on{exp}}
\def\abar{{\overline a}}
\def\xbar{{\overline x}}
\def\ybar{{\overline y}}
\def\zbar{{\overline z}}
\def\Rbar{{\overline R}}
\def\nbar{{\overline n}}
\def\cbar{{\overline c}}
\def\qbar{{\overline q}}
\def\bbar{{\overline b}}
\def\et{{\widetilde E}}
\def\ni{\noindent}
\def\coef{\on{coef}}
\def\den{\on{den}}
\def\lcm{\on{l.c.m.}}
\def\vi{v_1^{-1}}
\def\ot{\otimes}
\def\psibar{{\overline\psi}}
\def\thbar{{\overline\theta}}
\def\mhat{{\hat m}}
\def\exc{\on{exc}}
\def\ms{\medskip}
\def\ehat{{\hat e}}
\def\etao{{\eta_{\text{od}}}}
\def\etae{{\eta_{\text{ev}}}}
\def\dirlim{\operatornamewithlimits{dirlim}}
\def\gt{\widetilde{L}}
\def\lt{\widetilde{\lambda}}
\def\st{\widetilde{s}}
\def\ft{\widetilde{f}}
\def\sgd{\on{sgd}}
\def\lfl{\lfloor}
\def\rfl{\rfloor}
\def\ord{\on{ord}}
\def\gd{{\on{gd}}}
\def\rk{{{\on{rk}}_2}}
\def\nbar{{\overline{n}}}
\def\MC{\on{MC}}
\def\lg{{\on{lg}}}
\def\cB{\mathcal{B}}
\def\cS{\mathcal{S}}
\def\cP{\mathcal{P}}
\def\N{{\Bbb N}}
\def\Z{{\Bbb Z}}
\def\Q{{\Bbb Q}}
\def\R{{\Bbb R}}
\def\C{{\Bbb C}}
\def\l{\left}
\def\r{\right}
\def\mo{\on{mod}}
\def\xt{\times}
\def\notimm{\not\subseteq}
\def\Remark{\noindent{\it  Remark}}
\def\kut{\widetilde{KU}}
\def\hti{\widetilde{h}}
\def\lg{\operatorname{lg}}

\def\*#1{\mathbf{#1}}
\def\0{$\*0$}
\def\1{$\*1$}
\def\22{$(\*2,\*2)$}
\def\33{$(\*3,\*3)$}
\def\ss{\smallskip}
\def\ssum{\sum\limits}
\def\dsum{\displaystyle\sum}
\def\la{\langle}
\def\ra{\rangle}
\def\on{\operatorname}
\def\od{\text{od}}
\def\ev{\text{ev}}
\def\o{\on{o}}
\def\U{\on{U}}
\def\lg{\on{lg}}
\def\a{\alpha}
\def\bz{{\Bbb Z}}
\def\eps{\varepsilon}
\def\bc{{\bold C}}
\def\bN{{\bold N}}
\def\nut{\widetilde{\nu}}
\def\tfrac{\textstyle\frac}
\def\b{\beta}
\def\G{\Gamma}
\def\g{\gamma}
\def\zt{{\Bbb Z}_2}
\def\pt{\widetilde{p}}
\def\zth{{\bold Z}_2^\wedge}
\def\bs{{\bold s}}
\def\bx{{\bold x}}
\def\bof{{\bold f}}
\def\bq{{\bold Q}}
\def\be{{\bold e}}
\def\lline{\rule{.6in}{.6pt}}
\def\xb{{\overline x}}
\def\xbar{{\overline x}}
\def\ybar{{\overline y}}
\def\zbar{{\overline z}}
\def\ebar{{\overline \be}}
\def\nbar{{\overline n}}
\def\rbar{{\overline r}}
\def\Mbar{{\overline M}}
\def\et{{\widetilde e}}
\def\ni{\noindent}
\def\ms{\medskip}
\def\ehat{{\hat e}}
\def\what{{\widehat w}}
\def\Yhat{{\widehat Y}}
\def\nbar{{\overline{n}}}
\def\minp{\min\nolimits'}
\def\mul{\on{mul}}
\def\N{{\Bbb N}}
\def\Z{{\Bbb Z}}
\def\Q{{\Bbb Q}}
\def\R{{\Bbb R}}
\def\C{{\Bbb C}}
\def\notint{\cancel\cap}
\def\se{\operatorname{secat}}
\def\cS{\mathcal S}
\def\cR{\mathcal R}
\def\el{\ell}
\def\TC{\on{TC}}
\def\dstyle{\displaystyle}
\def\ds{\dstyle}
\def\mt{\widetilde{\mu}}
\def\zcl{\on{zcl}}
\def\Vb#1{{\overline{V_{#1}}}}

\def\Remark{\noindent{\it  Remark}}
\title[A lower bound for higher topological complexity]
{A lower bound for higher topological complexity of real projective space}
\author{Donald M. Davis}
\address{Department of Mathematics, Lehigh University\\Bethlehem, PA 18015, USA}
\email{dmd1@lehigh.edu}
\date{September 19, 2017}


\maketitle
\begin{abstract} We obtain an explicit formula for the best lower bound for the higher topological complexity, $\TC_k(RP^{n})$, of real projective space implied by mod 2 cohomology.
 \end{abstract}
\section{Main theorem}\label{intro}
The notion of higher topological complexity, $\TC_k(X)$, of a topological space $X$ was introduced in \cite{Rud}. It can be thought of as one less than the minimal number of rules required to tell how to move consecutively between any $k$ specified points of $X$. In \cite{5}, the study of $\TC_k(P^n)$ was initiated, where $P^n$ denotes real projective space.
Using $\zt$  coefficients for all cohomology groups, define $\zcl_k(X)$ to be the maximal number of elements in $\ker(\Delta^*:H^*(X)^{\ot k}\to H^*(X))$ with nonzero product.
It is standard that $$\TC_k(X)\ge\zcl_k(X).$$ In \cite{5}, it was shown that
$$\zcl_k(P^n)=\max\{a_1+\cdots +a_{k-1}:(x_1+x_k)^{a_1}\cdots(x_{k-1}+x_k)^{a_{k-1}}\ne0\}$$
in $\zt[x_1,\ldots,x_k]/(x_1^{n+1},\ldots,x_k^{n+1}).$
In Theorem \ref{mainthm} we give an explicit formula for $\zcl_k(P^n)$, and hence a lower bound for $\TC_k(P^n)$.

Our main theorem, \ref{mainthm}, requires some specialized notation.
\begin{defin} If $n=\sum \eps_j2^j$ with $\eps_j\in\{0,1\}$ (so the numbers $\eps_j$ form the binary expansion of $n$), let $$Z_i(n)=\ds\sum_{j=0}^i(1-\eps_j)2^j,$$
and let $$S(n)=\{i:\eps_i=\eps_{i-1}=1\text{ and }\eps_{i+1}=0\}.$$
\end{defin}
\noindent Thus $Z_i(n)$ is the sum of the 2-powers $\le 2^i$ which correspond to the 0's in the binary expansion of $n$. Note that $Z_i(n)=2^{i+1}-1-(n\text{ mod }2^{i+1})$.
The $i$'s in $S(n)$ are those that begin a sequence of two or more consecutive 1's in the binary expansion of $n$.
Also, $\nu(n)=\max\{t:2^t\text{ divides }n\}$.
\begin{thm}\label{mainthm} For $n\ge0$ and $k\ge3$,
\begin{equation}\label{maineq}\zcl_k(P^n)=kn-\max\{2^{\nu(n+1)}-1,2^{i+1}-1-k\cdot Z_i(n):i\in S(n)\}.
\end{equation}
\end{thm}

 It was shown in \cite{5} that, if $2^e\le n<2^{e+1}$, then $\zcl_2(P^n)=2^{e+1}-1$, which follows immediately from our Theorem \ref{thm1}.

 In Table \ref{T1}, we tabulate $\zcl_k(P^n)$ for $1\le n\le 17$ and $2\le k\le8$.

 \begin{table}[H]
\caption{Values of $\zcl_k(n)$}
\label{T1}
\begin{tabular}{c|ccccccccccccccccc}
$n$&$1$&$2$&$3$&$4$&$5$&$6$&$7$&$8$&$9$&$10$&$11$&$12$&$13$&$14$&$15$&$16$&$17$\\
\hline
$\zcl_2(n)$&$1$&$3$&$3$&$7$&$7$&$7$&$7$&$15$&$15$&$15$&$15$&$15$&$15$&$15$&$15$&$31$&$31$\\
$\zcl_3(n)$&$2$&$6$&$6$&$12$&$14$&$14$&$14$&$24$&$26$&$30$&$30$&$30$&$30$&$30$&$30$&$48$&$50$\\
$\zcl_4(n)$&$3$&$8$&$9$&$16$&$19$&$21$&$21$&$32$&$35$&$40$&$41$&$45$&$45$&$45$&$45$&$64$&$67$\\
$\zcl_5(n)$&$4$&$10$&$12$&$20$&$24$&$28$&$28$&$40$&$44$&$50$&$52$&$60$&$60$&$60$&$60$&$80$&$84$\\
$\zcl_6(n)$&$5$&$12$&$15$&$24$&$29$&$35$&$35$&$48$&$53$&$60$&$63$&$72$&$75$&$75$&$75$&$96$&$101$\\
$\zcl_7(n)$&$6$&$14$&$18$&$28$&$34$&$42$&$42$&$56$&$62$&$70$&$74$&$84$&$90$&$90$&$90$&$112$&$118$\\
$\zcl_8(n)$&$7$&$16$&$21$&$32$&$39$&$48$&$49$&$64$&$71$&$80$&$85$&$96$&$103$&$105$&$105$&$128$&$135$
\end{tabular}
\end{table}

The smallest value of $n$ for which two values of $i$ are significant in (\ref{maineq}) is $n=102=2^6+2^5+2^2+2^1$. With $i=2$, we have $7-k$ in the max, while with $i=6$, we have $127-25k$. Hence
$$\zcl_k(P^{102})=102k-\begin{cases}127-25k&2\le k\le 5\\
7-k&5\le k\le7\\
0&7\le k.\end{cases}$$

For all $k$ and $n$, $\TC_k(P^n)\le kn$ for dimensional reasons (\cite[Prop 2.2]{5}). Thus we obtain a sharp result $\TC_k(P^n)=kn$ whenever $\zcl_k(P^n)=kn$. Corollary \ref{cor} tells exactly when this is true.  Here is a simply-stated partial result.
\begin{prop}\label{lprop} If $n$ is even, then $\TC_k(P^n)=kn$ for $k\ge2^{\ell+1}-1$, where $\ell$ is the length of the longest string of consecutive $1$'s in the binary expansion of $n$.\end{prop}
\begin{proof} We use Theorem \ref{mainthm}. We need to show that if $i\in S(n)$ begins a string of $j$ 1's with $j\le\ell$, then $2^{i+1}-1\le(2^{\ell+1}-1)Z_i(n)$. If $j<\ell$, then $Z_i(n)\ge2^{i-j}+1$, and the desired inequality reduces to $2^{i+1}+2^{i-j}\le2^{\ell+1+i-j}+2^{\ell+1}$, which is satisfied since $2^{\ell+1+i-j}$ is strictly greater than both $2^{i+1}$ and $2^{i-j}$.

If $j=\ell$, then
$$Z_i(n)\ge1+\sum_\a2^{i+1-\a(\ell+1)},$$
where $\a$ ranges over all positive integers such that $i+1-\a(\ell+1)>0$. This reflects the fact that the binary expansion of $n$ has a 0 starting in the $2^{i-\ell}$ position and at least every $\ell+1$ positions back from there, and also a 0 at the end since $n$ is even. The desired inequality follows easily from this.\end{proof}
Theorem \ref{mainthm} shows that $\zcl_k(P^n)<kn$ when $n$ is odd.
In the next proposition, we give complete information about when $\zcl_k(n)=kn$ if $k=3$ or $4$.
\begin{prop}\label{kprop} If $k=3$ or $4$, then $\zcl_k(P^n)=kn$ if and only if $n$ is even and the binary expansion of $n$ has no consecutive $1$'s.\end{prop}
Proposition \ref{kprop} follows easily from Theorem \ref{mainthm} and the fact that if $i\in S(n)$, then $Z_i(n)\le2^{i-1}-1$.

The following recursive formula for $\zcl_k(P^n)$, which is interesting in its own right, is central to the proof of Theorem \ref{mainthm}.
It will be proved in Section \ref{sec2}.
\begin{thm}\label{thm1} Let $n=2^e+d$ with $0\le d<2^e$, and $k\ge2$.
If $z_k(n)=\zcl_k(P^n)$,
 then
$$z_k(n)=\min(z_k(d)+k2^e,(k-1)(2^{e+1}-1)), \text{ with }z_k(0)=0.$$
Equivalently, if $g_k(n)=kn-\zcl_k(P^n)$, then
\begin{equation}\label{zeq}g_k(n)=\max(g_k(d),kn-(k-1)(2^{e+1}-1)), \text{ with }g_k(0)=0.\end{equation}
\end{thm}

We now use Theorem \ref{thm1} to prove Theorem \ref{mainthm}.
\begin{proof}[Proof of Theorem \ref{mainthm}] We will prove that $g_k(n)$ of Theorem \ref{thm1} satisfies \begin{equation}\label{gmain}g_k(n)=\max\{2^{\nu(n+1)}-1,2^{i+1}-1-k Z_i(n):i\in S(n)\}\end{equation}
if $k\ge3$, which is clearly equivalent to Theorem \ref{mainthm}.
The proof is by induction, using the recursive formula  (\ref{zeq}) for $g_k(n)$.
Let $n=2^e+d$ with $0\le d<2^e$.

{\bf Case 1:} $d=0$. Then $n=2^e$ and by (\ref{zeq}) we have $g_k(n)=\max(0,k2^e-(k-1)(2^{e+1}-1))$. If $e=0$, this equals $1$, while if $e>0$, it equals $0$, since $k\ge3$.
These agree with the claimed answer $2^{\nu(n+1)}-1$, since $S(2^e)=\emptyset$.

{\bf Case 2:} $0<d<2^{e-1}$. Here $\nu(n+1)=\nu(d+1)$, $S(n)=S(d)$, and $Z_i(n)=Z_i(d)$ for any $i\in S(d)$. Substituting (\ref{gmain}) with $n$ replaced by $d$ into (\ref{zeq}), we obtain
$$g_k(n)=\max\{2^{\nu(n+1)}-1,2^{i+1}-1-kZ_i(n):i\in S(n), kn-(k-1)(2^{e+1}-1)\}.$$
We will be done once we show that $kn-(k-1)(2^{e+1}-1)$ is $\le$ one of the other entries, and so may be omitted.
If $i$ is the largest element of $S(n)$, we will show that $kn-(k-1)(2^{e+1}-1)\le2^{i+1}-1-kZ_i(n)$, i.e., \begin{equation}\label{need}k n_i\le(k-1)(2^{e+1}-2^{i+1}),\end{equation} where $n_i=n-(2^{i+1}-1-Z_i(n))$ is the sum of the 2-powers in $n$ which are greater than $2^i$. The largest of these is $2^e$, and no two consecutive values of $i$ appear in this sum, hence $n_i\le \sum2^j$, taken over $j\equiv e\ (2)$ and $i+2\le j\le e$. If $k=3$, (\ref{need}) is true because the above description of $n_i$ implies that $3n_i\le 2(2^{e+1}-2^{i+1})$, while for larger $k$, it is true since $\frac k{k-1}<\frac32$. If $S(n)$ is empty, then $kn- (k-1)(2^{e+1}-1)\le 2^{\nu(n+1)}-1$ by a similar argument, since $n\le 2^e+2^{e-2}+2^{e-4}+\cdots$, so $3n\le 2(2^{e+1}-1)$, and values of $k>3$ follow as before.

{\bf Case 3:} $d\ge2^{e-1}$. If $e-1\in S(d)$, then it is replaced by $e$ in $S(n)$, while other elements of $S(d)$ form the rest of $S(n)$. If $e-1\not\in S(d)$, then $S(n)=S(d)\cup\{e\}$. If $i\in S(n)-\{e\}$, then $Z_i(n)=Z_i(d)$, so its contribution  to the set of elements whose max equals $g_k(n)$ is $2^{i+1}-1-kZ_i(n)$, as desired. For $i=e$, the claimed term is $2^{e+1}-1-kZ_e(n)=kn-(k-1)(2^{e+1}-1)$, which is present by the induction from (\ref{zeq}). If $e-1\in S(d)$, then the $i=e-1$ term in the max for $g_k(d)$ is $2^e-1-kZ_i(n)$ and contributes to $g_k(n)$ less than the term described in the preceding sentence, and hence cannot contribute to the max. The $2^{\nu(n+1)}-1$ term is obtained from the induction since $\nu(n+1)=\nu(d+1)$.
\end{proof}

The author wishes to thank Jesus Gonz\'alez for many useful suggestions.
\section{Recursive formulas}\label{sec2}
In this section, we prove Theorem \ref{thm1} and the following variant.
\begin{thm}\label{thm2} Let $n=2^e+d$ with $0\le d<2^e$, and $k\ge2$. If $h_k(n)=\zcl_k(P^n)-(k-1)n$, then
\begin{equation}h_k(n)=\min(h_k(d)+2^e,(k-1)(2^{e+1}-1-n)), \text{ with }h_k(0)=0.\label{heq}\end{equation}

\end{thm}

\begin{proof}[Proof of Theorems \ref{thm1} and \ref{thm2}] It is elementary to check that the  formulas for $z_k$, $g_k$, and $h_k$ are equivalent to one another.
We prove (\ref{heq}). We first look for nonzero monomials in $(x_1+x_k)^{a_1}\cdots(x_{k-1}+x_k)^{a_{k-1}}$ of the form $x_1^n\cdots x_{k-1}^nx_k^\ell$ with $\ell\le n$. Letting $a_i=n+b_i$, the analogue of $h_k(n)$ for such monomials is given by
\begin{equation}\label{hdef}\hti_k(n)=\max\{\sum_{i=1}^{k-1}b_i:\tbinom{n+b_1}n\cdots\tbinom{n+b_{k-1}}n\text{ is odd and }\sum_{i=1}^{k-1}b_i\le n\},\end{equation}
since $\sum b_i$ is the exponent of $x_k$.
We will begin by proving
\begin{equation}\label{hti}\hti_k(n)=\min(\hti_k(d)+2^e,(k-1)(2^{e+1}-1-n)).\end{equation}

For a nonzero integer $m$, let $Z(m)$ (resp.~$P(m)$) denote the set of 2-powers corresponding to the 0's (resp.~1's) in the binary expansion of $m$, with $Z(0)=P(0)=\emptyset$. By Lucas's Theorem, $\binom{n+b_i}n$ is odd iff $P(b_i)\subset Z(n)$. Note that the integers $Z_i(n)$ considered earlier are sums of elements of subsets of $Z(n)$.

For a multiset $S$, let $\|S\|$ denote the sum of its elements, and let
$$\phi(S,n)=\max\{\|T\|\le n:T\subset S\}.$$ Note that $\|Z(n)\|=2^{\lg(n)+1}-1-n$, where $\lg(n)=\lfloor\log_2(n)\rfloor$, ($\lg(0)=-1$).
 Let $Z(n)^j$ denote the multiset consisting of $j$ copies of $Z(n)$, and let $$m_j(n)=\phi(Z(n)^j,n).$$
Then, from (\ref{hdef}), we obtain the key equation $\hti_k(n)=m_{k-1}(n)$. Thus (\ref{hti}) follows from Lemma \ref{lem1} below.

\begin{lem}\label{lem1} If $n=2^e+d$ with $0\le d<2^e$, and $j\ge1$, then
$$m_j(n)=\min(m_j(d)+2^e,j(2^{e+1}-1-n)).$$
\end{lem}
\begin{proof} The result is clear if $j=1$ since $2^{e+1}-1-n<2^e$, so we assume $j\ge2$. Let $S\subset Z(d)^j$ satisfy $\|S\|=m_j(d)$.

First assume $d<2^{e-1}$. Then $2^{e-1}\in Z(n)$. Let $T=S\cup\{2^{e-1},2^{e-1}\}$. No other subset of $Z(n)^j$ can have larger sum than $T$ which is $\le n$ due to maximality of $\|S\|$ and the fact that the 2-powers in $Z(n)^j-Z(d)^j$ are larger than those in $Z(d)^j$. Thus $m_j(n)=m_j(d)+2^e$ in this case, and this is $\le j(2^{e+1}-1-n)=\|Z(n)^j\|$.

If, on the other hand,  $d\ge 2^{e-1}$, then $Z(d)^j=Z(n)^j$. If $\|Z(n)^j-S\|<2^e$, then let $T=Z(n)^j$
with $\|T\|=j(2^{e+1}-1-n)$, as large as it could possibly be, and less than $m_j(d)+2^e$.  Otherwise, since any multiset of 2-powers whose sum is $\ge 2^e$ has a subset whose sum equals $2^e$, we can let $T=S\cup V$, where $V$ is a subset of $Z(n)^j-S$ with $\|V\|=2^e$. As before, no subset of $Z(n)^j$ can have size greater than that.

\end{proof}

Now we wish to consider more general monomials.
We claim that for any multiset $S$ and positive integers $m$ and $n$,
\begin{equation}\phi(Z(m-1)\cup S, n)\le\phi(Z(m)\cup S, n)+1.\label{phi}\end{equation}
This follows from the fact that subtracting 1 from $m$ can affect $Z(m)$ by adding 1, or changing $1, 2,\ldots,2^{t-1}$ to $2^t$. These changes cannot add more than 1 to the largest subset of size $\le n$.
We show now that this  implies that $h_k(n)=m_{k-1}(n)=\hti_k(n)$, and hence (\ref{heq}) follows from (\ref{hti}).

Suppose that $x_1^{n-\eps_1}\cdots x_{k-1}^{n-\eps_{k-1}}x_k^\ell$ with $\eps_i\ge0$ and $\ell\le n$ is a nonzero monomial in the expansion of $(x_1+x_k)^{n+b_1}\cdots(x_{k-1}+x_k)^{n+b_{k-1}}$. We wish to show that $\sum b_i\le m_{k-1}(n)$. It follows from (\ref{phi}) that
$$\phi\bigl(\bigcup_{i=1}^{k-1}Z(n-\eps_i),n\bigr)\le\phi(Z(n)^{k-1},n)+\sum\eps_i=m_{k-1}(n)+\sum\eps_i.$$
The odd binomial coefficients $\binom{n+b_i}{n-\eps_i}$ imply that $P(b_i+\eps_i)\subset Z(n-\eps_i)$. Thus
\begin{equation}\label{phi2}\phi\bigl(\bigcup_{i=1}^{k-1}P(b_i+\eps_i),n\bigr)\le m_{k-1}(n)+\sum\eps_i.\end{equation}
Since $\|P(b_i+\eps_i)\|=b_i+\eps_i$ and $\sum(b_i+\eps_i)\le n$, the left hand side of (\ref{phi2}) equals $\sum(b_i+\eps_i)$, hence $\sum b_i\le m_{k-1}(n)$, as desired.
\end{proof}

\section{Examples and comparisons}\label{explsec}
In this section, we examine some special cases of our results (in Propositions \ref{eis} and \ref{3prop}) and make  comparisons with some work in \cite{5}.

The numbers $z_3(n)=\zcl_3(P^n)$ are 1 less than a sequence which was  listed by the author as A290649 at \cite{OEIS} in August 2017.
They can be characterized as in Proposition \ref{eis}, the proof of which is a straightforward application of the recursive formula
$$z_3(2^e+d)=\min(z_3(d)+3\cdot2^e,2(2^{e+1}-1))\text{ for }0\le d<2^e,$$
from Theorem \ref{thm1}.
\begin{prop}\label{eis} For $n\ge0$, $\zcl_3(n)$ is the largest even integer $z$ satisfying $z\le 3n$ and $\binom{z+1}n\equiv1\ (2)$.\end{prop}
\noindent We have not found similar characterizations for $z_k(n)$ when $k>3$.

In \cite[Thm 5.7]{5}, it is shown that our $g_k(n)$ in Theorem \ref{thm1} is a decreasing function of $k$, and achieves a stable value of $2^{\nu(n+1)}-1$ for sufficiently large $k$. They defined $s(n)$ to be the minimal value of $k$ such that $g_k(n)=2^{\nu(n+1)}-1$. We obtain a  formula for the precise value of $s(n)$ in our next result.

Let $S'(n)$ denote the set of integers $i$ such that the $2^i$ position begins a string of two or more consecutive 1's in the binary expansion of $n$ which stops prior to the $2^0$ position. For example,
$S'(187)=\{5\}$ since its binary expansion is $10111011$.
\begin{prop} Let $s(-)$ and $S'(-)$ be the functions just described.   Then
$$s(n)=\begin{cases}2&\text{if }n+1\text{ is a $2$-power}\\
3&\text{if }n+1\text{ is not a $2$-power and }S'(n)=\emptyset\\
\max\bigl\{\biggl\lceil\dfrac{2^{i+1}-2^{\nu(n+1)}}{Z_i(n)}\biggr\rceil:i\in S'(n)\bigr\}&\text{otherwise}.\end{cases}$$
\label{sprop}\end{prop}
\begin{proof} It is shown in \cite[Expl 5.8]{5} that $g_k(2^v-1)=2^v-1$ for all $k\ge2$, hence $s(2^v-1)=2$. This also follows readily from (\ref{zeq}).

If the binary expansion of $n$ has a string of $i+1$ 1's at the end and no other consecutive 1's (so that $S(n)=\{i\}$ in (\ref{maineq})), then $Z_i(n)=0$. Thus by (\ref{gmain}) $g_k(n)=2^{i+1}-1=2^{\nu(n+1)}-1$ for $k\ge3$. If $n\ne2^{i+1}-1$, then $s(n)=3$, since  $g_2(n)>2^{i+1}-1$.

Now assume $S'(n)$ is nonempty. By (\ref{gmain}), $s(n)$ is the smallest $k$ such that \begin{equation}\label{disp}2^{i+1}-1-kZ_i(n)\le 2^{\nu(n+1)}-1\end{equation} for all $i\in S(n)$, which easily reduces to the claimed value.
 Note that if the string of 1's beginning at position $2^i$ goes all the way to the end, then (\ref{disp}) is satisfied; this case is omitted from $S'(n)$ in the theorem, because it would yield $0/0$.
\end{proof}

The following corollary is immediate.
\begin{cor}\label{cor} If $n$ is even and
$$k\ge\max\{3,\biggl\lceil\frac{2^{i+1}-1}{Z_i(n)}\biggr\rceil:i\in S(n)\},$$
then $\TC_k(P^n)=kn$. These are the only values of $n$ and $k$ for which $\zcl_k(P^n)=kn$.\end{cor}

In \cite[Def 5.10]{5}, a complicated formula was presented for numbers $r(n)$, and in \cite[Thm 5.11]{5}, it was proved that $s(n)\le r(n)$. It was conjectured there that $s(n)=r(n)$. However, comparison of the formula for $s(n)$ established in Proposition \ref{sprop} with their formula for $r(n)$ showed that there are many values of $n$ for which $s(n)<r(n)$. The first is $n=50$, where we prove $s(50)=5$, whereas their $r(50)$ equals 7.  Apparently their computer program did not notice that $$(x_1+x_5)^{63}(x_2+x_5)^{63}(x_3+x_5)^{62}(x_4+x_5)^{62}$$ contains the nonzero monomial $x_1^{50}x_2^{50}x_3^{50}x_4^{50}x_5^{50}$, showing that our $z_5(50)=250$ and $g_5(50)=0$, so $s(50)\le 5$.

In Table \ref{T2}, we present a table of some values of $s(-)$, omitting $s(2^v-1)=2$ and $s(2^v)=3$ for $v>0$.

\begin{table}[H]
\caption{Some values of $s(n)$}
\label{T2}
\begin{tabular}{c|cccccccccccccccccccc}
$n$&$5$&$6$&$9$&$10$&$11$&$12$&$13$&$14$&$17$-$21$&$22$&$23$&$24$&$25$&$26$&$27$&$28$&$29$&$30$\\
$s(n)$&$3$&$7$&$3$&$3$&$3$&$5$&$7$&$15$&$3$&$7$&$3$&$5$&$5$&$7$&$7$&$11$&$15$&$31$
\end{tabular}
\end{table}

In \cite{5}, there seems to be particular interest in $\TC_k(P^{3\cdot2^e})$. We easily read off from Theorem \ref{mainthm}  the following result.
\begin{prop}\label{3prop} For $k\ge2$ and $e\ge1$, we have
$$\zcl_k(P^{3\cdot2^e})=\begin{cases}(k-1)(2^{e+2}-1)&\text{if }(e=1,k\le6)\text{ or }(e\ge2,k\le4)\\
k\cdot3\cdot2^e&\text{otherwise.}\end{cases}$$\end{prop}
This shows that the estimate $s(3\cdot2^e)\le5$ for $e\ge2$ in \cite{5} is sharp.

 \def\line{\rule{.6in}{.6pt}}

\end{document}